\documentclass[a4paper, reqno, 12pt]{amsart}

\usepackage[utf8]{inputenc} 
\usepackage[T1, T2A]{fontenc}
\usepackage[english]{babel}
\usepackage{amsmath, amssymb, amsfonts, amsthm, amscd, mathrsfs}
\usepackage{enumitem}
\usepackage{tikz}
\usepackage[hidelinks]{hyperref}
\usepackage{forloop}
\usetikzlibrary{matrix}
\usetikzlibrary{calc}
\usepackage{dutchcal}

\setlist[enumerate]{nolistsep, label = (\alph*), ref=(\text{\alph*)}}
\setlist[itemize]{nolistsep}

\renewcommand{\Im}{\mathop{\text{Im}}}
\renewcommand{\phi}{\varphi}
\renewcommand{\ge}{\geqslant}
\renewcommand{\le}{\leqslant}

\newcommand{\KK}{\mathbb{K}}
\newcommand{\LL}{\mathbb{L}}
\renewcommand{\AA}{\mathbb{A}}
\newcommand{\QQ}{\mathbb{Q}}
\newcommand{\ZZ}{\mathbb{Z}}
\newcommand{\GG}{\mathbb{G}}

\newcommand{\g}{\mathcal{g}}
\newcommand{\R}{R(\g)}
\newcommand{\X}{X(\g)}

\newcommand{\NN}{\mathbb{Z}_{>0}}
\newcommand{\Zgezero}{\mathbb{Z}_{\geqslant 0}}
\newcommand{\divid}{\mid}
\newcommand{\rep}{\widehat}
\newcommand{\rrep}{\widetilde}
\newcommand{\pa}{\partial}
\newcommand{\rpa}{\rep \partial}
\newcommand{\dCb}{\delta_{C, \beta}}
\newcommand{\lici}{{l_{ic_i}}}
\newcommand{\pr}{^\prime}
\newcommand*{\hm}[1]{#1\nobreak\discretionary{}{\hbox{$\mathsurround=0pt #1$}}{}}

\DeclareMathOperator{\Ker}{Ker}
\DeclareMathOperator{\rdeg}{\deg^\prime}

\theoremstyle{plain}
\newtheorem{lemma}{Lemma}
\newtheorem{proposition}{Proposition}
\newtheorem{theorem}{Theorem}
\newtheorem{corollary}{Corollary}
\theoremstyle{definition}
\newtheorem{definition}{Definition}
\newtheorem{construction}{Construction}
\newtheorem{example}{Example}

\theoremstyle{remark}
\newtheorem{remark}{Remark}

\begin{document}

\title[LOCALLY NILPOTENT DERIVATIONS OF TRINOMIAL ALGEBRAS]{On homogeneous locally nilpotent derivations of trinomial algebras}
\author{Sergey Gaifullin}
\author{Yulia Zaitseva}
\date{}
\address{Lomonosov Moscow State University, Faculty of Mechanics and Mathematics, Department of Higher Algebra, Leninskie Gory 1, Moscow, 119991 Russia; \linebreak and \linebreak
National Research University Higher School of Economics, Faculty of Computer Science, Kochnovskiy Proezd 3, Moscow, 125319 Russia}
\email{sgayf@yandex.ru}
\address{Lomonosov Moscow State University, Faculty of Mechanics and Mathematics, Department of Higher Algebra, Leninskie Gory 1, Moscow, 119991 Russia;}
\email{yuliazaitseva@gmail.ru}

\subjclass[2010]{Primary 13N15, 14R20; \ Secondary 13A50, 14J50}
\thanks{The work was supported by the Foundation for the Advancement of Theoretical Physics and Mathematics ``BASIS''}
\keywords{Affine hypersurface, torus action, graded algebra, derivation}

\begin{abstract}
\noindent We provide an explicit description of homogeneous locally nilpotent derivations of the algebra of regular functions on affine trinomial hypersurfaces. As an application, we describe the set of roots of trinomial algebras. 
\end{abstract}

\maketitle

\section{Introduction}

Let $\KK$ be an algebraically closed field of characteristic zero and $R$ be an algebra over~$\KK$. 
A $\KK$-linear map $\delta\colon R \to R$ is called a derivation of the algebra $R$ if it satisfies the Leibniz rule, that is $\delta(fg) = \delta(f)g + f \delta(g)$ for any $f, g \in R$. A derivation $\delta$ is said to be locally nilpotent if for any $f \in R$ there exists $k \in \Zgezero$ such that $\delta^k(f) = 0$. 

Let $K$ be an abelian group. An algebra $R$ is said to be $K$-graded if 
$$
R = \bigoplus \limits_{w \in K} R_w
$$
and $R_{w_1}R_{w_2} \subseteq R_{w_1+w_2}$ for any $w_1, w_2 \in K$. 
Elements of $R_w$ are called homogeneous of degree~$w \in K$ and $R_w$ is called the homogeneous component of $R$ of degree~$w$. 
The weight monoid is the set $S = \{w \in K \mid R_w \ne 0\}$, the weight cone is the convex cone in the rational vector space $K_\QQ = K \otimes_\ZZ \QQ$ generated by~$S$. 
Every element of $R$ has a decomposition into the sum of homogeneous elements, which are called its homogeneous components. 
An ideal $I$ in $R$ is said to be homogeneous if the homogeneous components of any element of $I$ belong to $I$ as well. 
In particular, any principal ideal generated by a homogeneous element is homogeneous. 
If $I$ is a homogeneous ideal in $R$, then $R/I$ inherits the grading from~$R$. 

A derivation $\delta$ of $R$ is called homogeneous if it maps homogeneous elements of $R$ to homogeneous ones. 
In this case, there exists an element $w_0 \in K$ such that $\delta(R_w) \subseteq R_{w+w_0}$ for any $w \in K$. 
The element $w_0$ is called the degree of $\delta$. 

These notions have a geometric interpretation. 
Denote by $X$ an irreducible algebraic variety over~$\KK$ and by $\GG_a = (\KK, +)$ the additive group of the ground field. 
It is known that locally nilpotent derivations of the algebra $\KK[X]$ of regular functions on~$X$ are in one-to-one correspondence with regular $\GG_a$-actions on~$X$, see e.g.~\cite[Section 1.5]{Fr2006}.
Let $H$ be a quasitorus, that is a direct product of a torus and a finite abelian group. 
Suppose that $H$ acts on the variety~$X$. 
Then we have a corresponding grading on the algebra~$\KK[X]$ by the group of characters~$K$ of the quasitorus~$H$: 
$$
\KK[X] = \bigoplus \limits_{w \in K} \KK[X]_w, \quad \text{where}\quad \KK[X]_w = \{f \mid h \cdot f = w(h)f \quad\forall h \in H\}.
$$
It can be proved that a locally nilpotent derivation of~$\KK[X]$ is homogeneous with respect to this grading 
if and only if the quasitorus~$H$ normalizes the corresponding $\GG_a$-action on~$X$. 
A description of homogeneous locally nilpotent derivations enables us to describe the automorphism group of an algebraic variety in some cases, see e.g.~\cite[Theorem~5.5]{ArHaHeLi2012}, \cite[Theorem~5.5]{ArGa2017}. 

Let us define a trinomial algebra. 
Fix $n_0, n_1, n_2 \in \NN$ and denote $n = n_0 + n_1 + n_2$. 
For each $i = 0, 1, 2$, fix a tuple $l_i = (l_{i1}, \ldots, l_{in_i}) \in \NN^{n_i}$ and define a monomial 
$$
T_i^{l_i} = T_{i1}^{l_{i1}} \ldots T_{in_i}^{l_{in_i}} 
\in \KK[T_{ij} \mid i = 0, 1, 2, \; j = 1, \ldots, n_i]. 
$$
A polynomial of the form $\g = T_0^{l_0} + T_1^{l_1} + T_2^{l_2} \in \KK[T_{ij}]$ is called a trinomial, the hypersurface $\X$ given by $\{\g = 0\}$ in the affine space $\AA^n$ is a trinomial hypersurface, and the algebra $\R := \KK[T_{ij}]\,/\,(\g)$ of regular functions on $\X$ is a trinomial algebra. 

Our motivation to study trinomials comes from toric geometry. 
Consider an effective action ${T \times X \to X}$ of an algebraic torus~$T$ on an irreducible variety~$X$. 
The complexity of such an action is the codimension of a general $T$-orbit in $X$ and equals $\dim X - \dim T$. 

Actions of complexity zero are actions with an open $T$-orbit. 
A normal variety admitting a torus action with an open orbit is called a toric variety. 
Let $X$ be a toric variety. 
Then $\GG_a$-actions on~$X$ normalized by~$T$ (or, equivalently, locally nilpotent derivations of $\KK[X]$ that are homogeneous with respect to the corresponding grading) can be described in terms of Demazure roots of the fan corresponding to~$X$; see~\cite{De1970}, \cite[Section~3.4]{Od1988} for the original approach and~\cite{Li22010, ArLi2012, ArHaHeLi2012} for generalizations. 

Let $X$ admit a torus action of complexity one. 
A description of $\GG_a$-actions on~$X$ normalized by~$T$ (or homogeneous locally nilpotent derivations of $\KK[X]$) in terms of proper polyhedral divisors may be found in~\cite{Li12010} and~\cite{Li22010}. 
It is an interesting problem to find their explicit form in concrete cases. 

The study of toric varieties is related to binomials, see e.g.~\cite[Chapter~4]{St1996}. 
At the same time, Cox rings establish a close relation between torus actions of complexity one and trinomials, see~\cite{HaSu2010, HaHeSu2011, HaHe2013, ArHaHeLi2012, HaWr2017}. 
In particular, any trinomial hypersurface admits a torus action of complexity one.

We consider the fine grading on a trinomial algebra~$\R$, it is defined in Construction~\ref{gr_constr}. 
This grading corresponds to an effective action of a quasitorus on the trinomial hypersurface~$\X$. 
The neutral component of the quasitorus is a torus of dimension~$n - 2$, which acts on~$\X$ with complexity one. 

In this paper, homogeneous locally nilpotent derivations of a trinomial algebra~$\R$ are described. 
It is proved that they are elementary (Theorem~\ref{theor}), see Section~\ref{sect_dCb} for definitions. 
In~\cite[Theorem~4.3]{ArHaHeLi2012}\footnote{In~\cite{ArHaHeLi2012}, a more general class of affine varieties defined by a system of trinomials is studied, see~\cite[Construction~3.1]{ArHaHeLi2012} for details.}, 
this statement was proved for primitive derivations, i.e. for homogeneous derivations with the degrees not belonging to the weight cone of~$\R$. 
Every primitive derivation is locally nilpotent; the converse is false (see~\cite[Example~7]{Za2017}). 
In~\cite{Za2017}, Theorem~\ref{theor} was proved for some classes of trinomial algebras~$\R$, including non-factorial algebras~$\R$.  
In this paper, we use this result and obtain a description in a general case. 
Specifically, we reduce the remaining case to a description of homogeneous locally nilpotent derivations for $X = \{x+y+z^k = 0\}$, see Proposition~\ref{pr_23m} and Lemma~\ref{l_base}. 

The degrees of homogeneous locally nilpotent derivations are called the roots of the variety $X$. This definition imitates in some sense the notion of a root from Lie Theory. 

In~\cite{Li2011}, A.\,Liendo gives a description of roots of the affine Cremona group. This answers a question due to Popov, see~\cite{Po2005}. The proof is based on the classification of locally nilpotent derivations in the case of torus actions of complexity one given in~\cite{Li22010}. In~\cite{Ko22014}, this result is obtained more directly. 

Theorem~\ref{theor} enables us to find roots of~$\X$, see Section~\ref{sect_cor}. 
Besides that, Theorem~\ref{theor} gives a description of locally nilpotent derivations and roots in some important particular cases, see Examples~\ref{ex_dCb}, \ref{ex_pyu1}, and~\ref{ex_pyu2}. 

\smallskip

The authors are grateful to Ivan Arzhantsev for useful discussions and comments and to the referee for careful reading of the paper and valuable suggestions.

\section{Trinomial algebras}

In this section, the fine grading $\deg$ on a trinomial algebra is defined and some properties of homogeneous elements are proved. 

\begin{construction}
\label{gr_constr}
Fix $n_0, n_1, n_2 \in \NN$ and let $n = n_0 + n_1 + n_2$. 
Consider the polynomial algebra $\KK[T_{ij},\, i = 0, 1, 2,\, j = 1, \ldots, n_i] =: \KK[T_{ij}]$. 
For each $i = 0, 1, 2$, fix a tuple  
$l_i \hm= (l_{i1}, \ldots, l_{in_i}) \in \NN^{n_i}$ and define a monomial 
$T_i^{l_i} = T_{i1}^{l_{i1}} \ldots T_{in_i}^{l_{in_i}} \in \KK[T_{ij}]$.  
By a \textit{trinomial} we mean a polynomial of the form 
$$
\g = T_0^{l_0} + T_1^{l_1} + T_2^{l_2}.
$$
A \textit{trinomial hypersurface} is the zero set $\X = \{\g = 0\}$ in the affine space $\AA^n$. 
It can be checked that the polynomial $\g$ is irreducible, hence the 
algebra $\R := \KK[T_{ij}]\,/\,(\g)$ of regular functions on the trinomial hypersurface has no zero divisors. 
We call such algebras $\R$ \textit{trinomial}.  
We~use the same notation for elements of $\KK[T_{ij}]$ and their projections to~$\R$.

Following~\cite{HaHe2013}, we consider a ($2 \times n$)-matrix $L$ corresponding to the trinomial $\g$ as follows: 
$$L = 
\begin{pmatrix}
-l_0 & l_1 & 0 \\
-l_0 & 0 & l_2 
\end{pmatrix}.
$$
Let $L^*$ be the transpose of~$L$. 
Denote by $K$ the factor group $K = \ZZ^n\,/\,\Im L^*$ and by $Q \colon \ZZ^n \to K$ the canonical projection. 
Let $e_{ij} \in \ZZ^n$ be the canonical basis vectors. 
The equalities 
\begin{equation}
\label{def_deg}
\deg T_{ij} = Q(e_{ij})
\end{equation}
define a $K$-grading on the algebra~$\KK[T_{ij}]$.

Note that the sums $l_{i1} Q(e_{i1}) + \ldots + l_{in_i} Q(e_{in_i}) \in K$
are the same for $i = 0, 1, 2$. Hence $\g$ is a homogeneous polynomial and equalities~(\ref{def_deg}) define a $K$-grading on~$\R = \KK[T_{ij}]\,/\,(\g)$ as well. 
We call this grading \emph{the fine grading} and denote the degree with respect to this grading by $\deg$. The derivations that are homogeneous with respect to the fine grading are called \emph{finely homogeneous}. 
\end{construction}

\begin{example}
\label{ex_grad}
Let all $l_{ij} = 1$, that is 
$\g = T_{01}\ldots T_{0n_0} + T_{11}\ldots T_{1n_1} + T_{21} \ldots T_{2n_2}$. Then
$L = \begin{pmatrix}
-1 & \ldots & -1 & 1 & \ldots & 1 & 0 & \ldots & 0 \\ 
-1 & \ldots & -1 & 0 & \ldots & 0 & 1 & \ldots & 1 \end{pmatrix}$. 
The matrix $L$ can be reduced by elementary row and column operations to the form 
$\begin{pmatrix} 
1 & 0 & 0 & \ldots & 0 \\ 
0 & 1 & 0 & \ldots & 0 \end{pmatrix}$, 
hence the group $K = \ZZ^n\,/\,\Im L^*$ is isomorphic to $\ZZ^{n-2}$. Let us choose some basis $\left\{v_{ij}, \, (i, j) \notin \{(0, 1), (1, 1)\}\right\}$ of $\ZZ^{n-2}$. 
We can define the grading explicitly via
$$
\deg T_{01} = -\sum \limits_{k=2}^{n_0} v_{0k} + 
\sum \limits_{l=1}^{n_2} v_{2l}, \quad
\deg T_{11} = -\sum \limits_{k=2}^{n_1} v_{1k} + 
\sum \limits_{l=1}^{n_2} v_{2l}, \; \text{and }
\deg T_{ij}  = v_{ij} \text{\quad otherwise}.
$$
\end{example}

\smallskip

The following lemma states that the constructed grading $\deg$ is the finest grading on the algebra~$\R$ such that all generators~$T_{ij}$ are homogeneous. It follows easily from the definition of the grading $\deg$. 

\begin{lemma}
\label{l_rdeg}
Let $\deg$ be the fine grading on the algebra~$\R$ defined in Construction~\ref{gr_constr} and $\rdeg$ be any grading on~$\R$ by some abelian group such that all generators~$T_{ij}$ are homogeneous. 
Then $\rdeg = \psi \circ \deg$ for some homomorphism $\psi \colon K \to K^\prime$. 
In particular,
any derivation that is homogeneous with respect to the fine grading $\deg$ is homogeneous with respect to the grading $\rdeg$ as well. 
\end{lemma}

For any algebra $R$ graded by an abelian group, one can consider the subalgebra of degree zero fractions inside the field of fractions:
$$
Q(R)_0 = \left\{\frac{f}{g} \biggm| f, g \in R, \, \text{homogeneous}, \, \deg f = \deg g, \, g \ne 0\right\} \subseteq Q(R). 
$$

The following lemma is proved in~\cite[Proposition 3.4]{ArHaHeLi2012}.
\begin{lemma}
\label{fracfield}
Take any $i, j \in \{0, 1, 2\}, \, i \ne j$. Then the field of degree zero fractions of $\R$ is the rational function field
$$
Q(\R)_0 = \KK\left(\frac{T_i^{l_i}}{T_j^{l_j}}\right).
$$
\end{lemma}

\smallskip

For any $w \in K$, denote by $R_w$ the finely homogeneous component of $\R$ of degree $w$. 

\begin{lemma}
\label{lem_hom1}
Let $h \in \R$ be homogeneous. Then $h = \rrep h F(T_0^{l_0}, T_1^{l_1})$ for some homogeneous polynomial $F$ in two variables and some $\rrep h \in R_{\rrep w}$ with $\dim R_{\rrep w} = 1$.  
\end{lemma}

This follows from the proof of~\cite[Proposition 3.5]{ArHaHeLi2012}. 

\begin{lemma}
\label{lem_hom2}
Let $h \in \R$ be homogeneous. Then $h = \prod T_{ij}^{u_{ij}} \cdot F(T_0^{l_0}, T_1^{l_1})$ for some homogeneous polynomial $F$ in two variables and some $u_{ij} \in \Zgezero$. 
\end{lemma}

\begin{proof}
Let $\rrep h$ from Lemma~\ref{lem_hom1} be equal to 
$\rrep h = \sum \limits_{k = 1}^s h_k$, where $h_k \ne 0$ are non-proportional monomials in $T_{ij}$ and $\deg h_k = \deg \rrep h = \rrep w$. Then $\dim R_{\rrep w} = 1$ implies $\rrep h = h_1$. 
\end{proof}

\section{Elementary derivations}
\label{sect_dCb}

The following construction is given in~\cite{ArHaHeLi2012} and is described below in the case of trinomial hypersurfaces 
(in notation of~\cite{ArHaHeLi2012} that is $r = 2$, 
$A = \left(\begin{smallmatrix} 0 & -1 & 1 \\ 1 & -1 & 0 \end{smallmatrix}\right)$, 
$\g = g_I$ for $I = \{0, 1, 2\}$, $\R = R(A, P_0)$).
\begin{construction}
\label{dCb_constr} 
Let us define a derivation $\dCb$ of~$\R$. 
The input data are
\begin{itemize}
\item a sequence $C = (c_0, c_1, c_2)$, where $c_i \in \ZZ$, $1 \le c_i \le n_i$; 
\item a vector $\beta = (\beta_0, \beta_1, \beta_2)$ such that $\beta_i \in \KK$, $\beta_0 + \beta_1 + \beta_2 = 0$. 
\end{itemize}
It is clear that for $\beta \ne 0$ as above either all entries $\beta_i$ differ from zero or there is a unique~$i_0$ with $\beta_{i_0} = 0$. 
According to these cases, we put further conditions and define: 
\begin{enumerate}[label = (\roman*), ref=(\roman*)]
\item \label{dCbconstr_1}
if $\beta_i \ne 0$ for all $i = 0, 1, 2$ and there is at most one $i_1$ with $l_{i_1c_{i_1}} > 1$, then we set 
$$
\dCb(T_{ij}) = 
\begin{cases}
\beta_i \prod \limits_{k \ne i} \frac{\pa T_k^{l_k}}{\pa T_{kc_k}},    &j = c_i,\\
0, &j \ne c_i.
\end{cases}
$$ 
\item \label{dCbconstr_2}
if $\beta_{i_0} = 0$ for a unique $i_0$ and there is at most one $i_1$ with $i_1 \ne i_0$ and $l_{i_1c_{i_1}} > 1$, then 
$$
\dCb(T_{ij}) = 
\begin{cases}
\beta_i \prod \limits_{k \ne i, i_0} \frac{\pa T_k^{l_k}}{\pa T_{kc_k}},    &j = c_i,\\
0, &j \ne c_i.
\end{cases}
$$
\end{enumerate}
\end{construction}

These assignments define a map~$\dCb$ on the generators~$T_{ij}$.
It can be extended uniquely to a~derivation on~$\KK[T_{ij}]$ by Leibniz rule. 
One can check that $\dCb(\g) = 0$, whence constructed map induces a well-defined derivation of the factor algebra~$\R$.

\begin{lemma}
Every derivation of the form~$\dCb$ of the algebra~$\R$ is homogeneous and locally nilpotent. 
\end{lemma}

The proof is given in~\cite[Construction~4.2]{ArHaHeLi2012}.

\smallskip

Let $h \in \R$ be a homogeneous element in the kernel of a derivation $\dCb$. 
The derivation~$h \dCb$ is also locally nilpotent. 

\begin{definition}
We say that a derivation of a trinomial algebra~$\R$ is \textit{elementary} if it has the form~$h \dCb$, where $h$ is a homogeneous element in the kernel of~$\dCb$. 
In addition, \textit{elementary derivations of Type~I} are elementary derivations with $\dCb$ corresponding to case~\ref{dCbconstr_1} in Construction~\ref{dCb_constr}; \textit{elementary derivations of Type~II} are elementary derivations with $\dCb$ corresponding to case~\ref{dCbconstr_2}. 
\end{definition}

\begin{proposition}
\label{prop_ker}
Let $\dCb$ be finely homogeneous locally nilpotent derivation of $\R$ defined in Construction~\ref{dCb_constr}. Then $h$ is a homogeneous element of $\Ker \dCb$ if and only if   
$$
h = \alpha \!\!\!\! \prod \limits_{T_{ij} \in \Ker \dCb} \!\!\!\!\!\!\! T_{ij}^{u_{ij}} \cdot \left(\beta_1 T_0^{l_0} - \beta_0 T_1^{l_1}\right)^m,
$$
for some $u_{ij}, m \in \Zgezero$, $\alpha \in \KK$.
\end{proposition}

\begin{proof}
Let $h$ be homogeneous, $h \in \Ker \dCb$. 
According to Lemma~\ref{lem_hom2},
$$
h = \prod T_{ij}^{u_{ij}} \cdot F(T_0^{l_0}, T_1^{l_1})
$$
for some homogeneous polynomial $F$ in two variables and some $u_{ij} \in \Zgezero$. 
Every homogeneous polynomial in two variables over an algebraically closed field can be decomposed in a product of linear polynomials. 
Hence 
$$
h = \prod T_{ij}^{u_{ij}} \cdot \prod \limits_k \left(\zeta_k T_0^{l_0} + \xi_k T_1^{l_1}\right)
$$
for some $\zeta_k, \xi_k \in \KK$.

The kernel of a locally nilpotent derivation is factorially closed (see, for example, \cite[Principle 1]{Fr2006}). 
Thus $h \in \Ker \dCb$ implies that all $T_{ij}$ with $u_{ij} > 0$ in the first product and all $\left(\zeta_k T_0^{l_0} + \xi_k T_1^{l_1}\right)$ in the second product belong to $\Ker \dCb$ as well. 
By definition of $\dCb$, 
$$0 = \dCb\left(\zeta_k T_0^{l_0} + \xi_k T_1^{l_1}\right) = 
\begin{cases}
(\zeta_k \beta_0 + \xi_k \beta_1) \prod \limits_{i} \frac{\pa T_i^{l_i}}{\pa T_{ic_i}} & \text{in the case~\ref{dCbconstr_1},} \\
(\zeta_k \beta_0 + \xi_k \beta_1) \prod \limits_{i \ne i_0}\frac{\pa T_i^{l_i}}{\pa T_{ic_i}} & \text{in the case~\ref{dCbconstr_2}},
\end{cases}
$$
whence $\zeta_k \beta_0 + \xi_k \beta_1 = 0$. Consequently any pair $(\zeta_k, \xi_k)$ is proportional to $(\beta_1, -\beta_0)$. This yields that $h$ has the required form. 

Conversely, any $h$ of such a form is homogeneous and belongs to $\Ker \dCb$, since all factors are of this kind. 
\end{proof}

\begin{example}
\label{ex_dCb}
Suppose that all $l_{ij} = 1$, that is 
$\g = T_{01}\ldots T_{0n_0} + T_{11}\ldots T_{1n_1} + T_{21} \ldots T_{2n_2}$ (see Example~\ref{ex_grad}). Let us find all elementary derivations of $\R$. 
For any $C = (c_0, c_1, c_2)$, $1 \le c_i \le n_i$, we have $\lici = 1$, so both cases of Construction~\ref{dCb_constr} are possible. 
\begin{enumerate}[label = (\roman*), ref=(\roman*)]
\item 
Let $\beta = (\beta_0, \beta_1, \beta_2)$, $\beta_0 + \beta_1 + \beta_2 = 0$, $\beta_i \ne 0$ for all $i = 0, 1, 2$. 
Then we have $\dCb(T_{0c_0}) = \beta_0 T_1^{l_1} T_2^{l_2} / T_{1c_1}T_{2c_2}$ 
and two analogous formulas for $T_{1c_1}$ and $T_{2c_2}$, hence 
$$\dCb = \dfrac{T_0^{l_0}T_1^{l_1}T_2^{l_2}}{T_{0c_0}T_{1c_1} T_{2c_2}}
\left(
\beta_0 \dfrac{T_{0c_0}}{T_0^{l_0}} \dfrac{\pa}{\pa T_{0c_0}} +
\beta_1 \dfrac{T_{1c_1}}{T_1^{l_1}} \dfrac{\pa}{\pa T_{1c_1}} +
\beta_2 \dfrac{T_{2c_2}}{T_2^{l_2}} \dfrac{\pa}{\pa T_{2c_2}} 
\right).$$
\item
Consider for example the case $i_0 = 2$, that is $\beta = (\beta_0, -\beta_0, 0)$. 
Then we get $\dCb(T_{0c_0}) = \beta_0 T_1^{l_1} / {T_{1c_1}}$ and the same formula for $T_{1c_1}$, hence
$$
\dCb = \beta_0 \dfrac{T_0^{l_0}T_1^{l_1}}{T_{0c_0}T_{1c_1}}
\left(
\dfrac{T_{0c_0}}{T_0^{l_0}} \dfrac{\pa}{\pa T_{0c_0}} -
\dfrac{T_{1c_1}}{T_1^{l_1}} \dfrac{\pa}{\pa T_{1c_1}}
\right).$$
\end{enumerate}
Thus we have $n_0n_1n_2 + n_0n_1 + n_1n_2 + n_2n_0$ classes of $\dCb$. 
Applying Proposition~\ref{prop_ker} we obtain that every elementary derivation $\delta$ has the form $\delta = h\dCb$, where a homogeneous element $h$ belongs to the kernel of corresponding $\dCb$. 
In the above cases, an elementary derivation $\delta$ is of the following forms for some $u_{ij}, m \in \Zgezero$: 
\begin{gather*}
\delta = \!\!\!\!\prod \limits_{(i, j) \ne (i, c_i)} \!\!\!\!\! T_{ij}^{u_{ij}} \cdot \left(\beta_1 T_0^{l_0} - \beta_0 T_1^{l_1}\right)^m
\dfrac{T_0^{l_0}T_1^{l_1}T_2^{l_2}}{T_{0c_0}T_{1c_1} T_{2c_2}}
\left(
\beta_0 \dfrac{T_{0c_0}}{T_0^{l_0}} \dfrac{\pa}{\pa T_{0c_0}} +
\beta_1 \dfrac{T_{1c_1}}{T_1^{l_1}} \dfrac{\pa}{\pa T_{1c_1}} +
\beta_2 \dfrac{T_{2c_2}}{T_2^{l_2}} \dfrac{\pa}{\pa T_{2c_2}} 
\right);
\\
\delta = \alpha \!\!\!\!\prod \limits_{\substack{(i, j) \ne (0, c_0) \\ (i, j) \ne (1, c_1)}} \!\!\!\!\! T_{ij}^{u_{ij}} \cdot \left(T_0^{l_0} + T_1^{l_1}\right)^m
\dfrac{T_0^{l_0}T_1^{l_1}}{T_{0c_0}T_{1c_1}}
\left(
\dfrac{T_{0c_0}}{T_0^{l_0}} \dfrac{\pa}{\pa T_{0c_0}} -
\dfrac{T_{1c_1}}{T_1^{l_1}} \dfrac{\pa}{\pa T_{1c_1}}
\right).
\end{gather*}
\end{example}

\section{Main results} 

The following theorem is the main result of the paper. 

\begin{theorem}
\label{theor}
Every finely homogeneous locally nilpotent derivation of a trinomial algebra is elementary. 
\end{theorem}

For example, all finely homogeneous locally nilpotent derivations of the trinomial algebra~$\R$ for $\g = T_{01}\ldots T_{0n_0} + T_{11}\ldots T_{1n_1} + T_{21} \ldots T_{2n_2}$ are described in Example~\ref{ex_dCb}. 

\smallskip

Let us start to prove Theorem~\ref{theor}.

\smallskip

The following lemma includes basic properties of locally nilpotent derivations. The proof can be found, for example, in~\cite[Principle 5 and Corollary 1.20]{Fr2006}. 

\begin{lemma} 
\label{fr_lem}
Suppose $\delta\colon R \to R$ is a locally nilpotent derivation of a domain $R$ and $f, g \in R$. 
\begin{enumerate}
\item \label{fr_lem_f} 
If $f \divid \delta(f)$, then $\delta(f) = 0$.
\item \label{fr_lem_fg}
If $f \divid \delta(g)$ and $g \divid \delta(f)$, then $\delta(f) = 0$ or $\delta(g) = 0$. 
\end{enumerate}
\end{lemma}

\begin{definition}
Let $\delta$ be a derivation of a trinomial algebra~$\R$. We say that a variable~$T_{ij}$ is a \textit{non-kernel variable} if $\delta(T_{ij}) \ne 0$. The exponent $l_{ij}$ of such variable is called a \textit{non-kernel exponent}.
\end{definition}

The following lemma is proved in~\cite[Lemma~3.4]{Ga2017}. 
For convenience of the reader we give a short proof below. 

\begin{lemma}
\label{3_lemma}
Let $\delta$ be a finely homogeneous locally nilpotent derivation of an algebra~$\R$. 
Then there exists at most one non-kernel variable in every monomial~$T_i^{l_i}$. 
\end{lemma}

\begin{proof}
Assume the converse. Then there is a monomial with at least two non-kernel variables. 
We can assume that these variables are $T_{01}$ and $T_{02}$ in $T_0^{l_0}$. 

Consider the following $\ZZ$-grading on the algebra~$\KK[T_{ij}]$:
$$
\rdeg T_{01} = l_{02}, \quad \rdeg T_{02} = -l_{01}, \quad \rdeg T_{ij} = 0 \;\text{for all other $(i, j)$.}
$$
The trinomial $\g$ is homogeneous (of degree $0$) with respect to this grading, therefore the \hbox{$\rdeg$-grading} is a well-defined grading on the factor algebra~$\R$. 
By Lemma~\ref{l_rdeg}, it follows that the derivation~$\delta$ is $\rdeg$-homogeneous. 
We have the following two cases.

1) $\rdeg \delta \ge 0$. 
Then $\rdeg \delta(T_{01}) = \rdeg \delta + \rdeg T_{01} > 0$. 
Note that $T_{01}$ is a unique variable with a positive degree. 
Hence every monomial in $\delta(T_{01})$ includes $T_{01}$ and therefore $T_{01}$ divides $\delta(T_{01})$. But $\delta(T_{01}) \ne 0$ by assumption. This contradicts Lemma~\ref{fr_lem}\ref{fr_lem_f}. 

2) $\rdeg \delta \le 0$. 
Then $\rdeg \delta(T_{02}) = \rdeg \delta + \rdeg T_{02} < 0$. 
In the same way, $T_{02}$ divides $\delta(T_{02})$ and 
this contradicts Lemma~\ref{fr_lem}\ref{fr_lem_f}.
\end{proof}

Proof of the following Proposition is given in~\cite[Proposition~1]{Za2017}. 

\begin{proposition}
\label{pr_prop}
Let $\delta \colon \R \to \R$ be a finely homogeneous locally nilpotent derivation. Suppose that $\delta(T_0^{l_0})$, $\delta(T_1^{l_1})$, and $\delta(T_2^{l_2})$ lie in a subspace of~$\R$ of dimension~$1$. Then the derivation~$\delta$ is elementary. 
\end{proposition}

The following proposition is in fact proved in~\cite[Theorem~1]{Za2017}. 
However, the required formulation is stronger than in~\cite[Theorem~1]{Za2017}, and for convenience of the reader a proof is given below. 

\begin{proposition}
\label{pr_01m'}
Let $\delta$ be a finely homogeneous locally nilpotent derivation of a trinomial algebra $\R$. Let at most one non-kernel exponent be equal to~$1$. Then $\delta(T_0^{l_0})$, $\delta(T_1^{l_1})$, and $\delta(T_2^{l_2})$ lie in a subspace of~$\R$ of dimension~$1$. 
\end{proposition}

\begin{proof}
According to Lemma~\ref{3_lemma}, there are at most three non-kernel variables. 
Suppose that there exist at most two non-kernel variables. 
Then at least one of $\delta(T_0^{l_0})$, $\delta(T_1^{l_1})$, and $\delta(T_2^{l_2})$ is equal to~$0$, and the assertion follows from $\delta(\g) = \delta(0) = 0$ in~$\R$. 
Hence we can assume without loss of generality that there are three non-kernel variables $T_{01}$, $T_{11}$, $T_{21}$ with non-kernel exponents $l_{01} > 1$, $l_{11} > 1$, $l_{21}$. 

Denote by $\ZZ_m = \{[0]_m, [1]_m, \ldots, [m - 1]_m\}$ the cyclic group of order $m$. Consider the following $\ZZ_m$-grading on $\KK[T_{ij}]$: 
$$
\rdeg T_{01} = [1]_{l_{01}}, \quad \rdeg T_{ij} = [0]_{l_{01}} \;\;\text{for all other $(i, j)$.}
$$
Since $\g$ is a homogeneous polynomial with respect to this grading (of degree $[0]_{l_{01}}$), we get a well-defined grading $\rdeg$ on~$\R$. 

If $\rdeg \delta(T_{01}) \ne [0]_{l_{01}}$ then $T_{01}$ divides $\delta(T_{01})$. 
Indeed, this inequality implies that the degrees of all monomials in $\delta(T_{01})$ do not equal~$[0]_{l_{01}}$, 
and it is possible only if every monomial includes $T_{01}$ 
(since $T_{01}$ is a unique variable with a nonzero degree). 
On the other hand, the fact that $T_{01}$ divides $\delta(T_{01})$ together with $\delta(T_{01}) \ne 0$ contradicts Lemma~\ref{fr_lem}\ref{fr_lem_f}. 
Thus $\rdeg \delta(T_{01}) = [0]_{l_{01}}$. 

Then $\rdeg \delta = \rdeg \delta(T_{01}) - \rdeg T_{01} = [l_{01} - 1]_{l_{01}} \ne [0]_{l_{01}}$ (recall that $l_{01} > 1$). 
It follows that $\rdeg \delta(T_{11}) = \rdeg \delta + \rdeg T_{11} \ne [0]_{l_{01}}$. Arguing as above, we conclude that $T_{01}$ divides~$\delta(T_{11})$. 

Similarly, $T_{11}$ divides~$\delta(T_{01})$. This contradicts Lemma~\ref{fr_lem}\ref{fr_lem_fg}. 
\end{proof}

The proof of the following proposition is given in the next section. 

\begin{proposition}
\label{pr_23m}
Let $\delta$ be a finely homogeneous locally nilpotent derivation of a trinomial algebra $\R$. Let at least two non-kernel exponents be equal to~$1$. Then $\delta(T_0^{l_0})$, $\delta(T_1^{l_1})$, and $\delta(T_2^{l_2})$ lie in a subspace of~$\R$ of dimension~$1$.  
\end{proposition}

\begin{proof}[Proof of Theorem~\ref{theor}]
The result follows from Propositions~\ref{pr_01m'}, \ref{pr_23m}, and~\ref{pr_prop}. 
\end{proof}

\section{Proof of Proposition~\ref{pr_23m}}

Recall that an element of polynomial algebra $v \in \KK[y, z]$ is a \textit{coordinate} if there is an element $u \in \KK[y, z]$ such that $\KK[u, v] = \KK[y, z]$. 
The proof of the following lemma can be found, for example, in~\cite[Corollary~4.7]{Fr2006}. 

\begin{lemma}
\label{jac}
Let $\pa$ be a derivation of~$\KK[y, z]$. Then $\pa$ is locally nilpotent if and only if $\pa$ is of the form $\pa(\,\cdot\,) = J(\,\cdot\,, g)$, where $J$ denotes the Jacobian and $g \in \KK[v]$ for some coordinate $v$ of $\KK[y, z]$.
\end{lemma}

First let us prove the simplest case of Proposition~\ref{pr_23m}. 

\begin{lemma}
\label{l_base}
Let $\g = x + y + z^k$. Let $\pa$ be a locally nilpotent finely homogeneous derivation of the algebra $\R$.
Then $\pa(x)$, $\pa(y)$, and $\pa(z^k)$ lie in a subspace of~$\R$ of dimension~$1$.
\end{lemma}

\begin{proof}
The fine grading in this case is given by
$$
\deg x = k, \quad \deg y = k, \quad \deg z = 1.
$$

The variety $\{x + y + z^k = 0\}$ in $\KK^3$ is isomorphic to the plane $\KK^2$ with coordinates $y$ and $z$, the isomorphism is given by $x = - y - z^k$. 
Hence we have the corresponding locally nilpotent derivation $\rpa$ of the polynomial algebra $\KK[y, z]$, which is homogeneous with respect to the grading
\begin{equation}
\label{tab}
\deg y = k, \quad \deg z = 1.
\end{equation}

By Lemma~\ref{jac}, $\rpa$ has the form $\rpa(\,\cdot\,) = J(\,\cdot\,, g(v))$, where $g$ is a polynomial in one variable over $\KK$ and $v$ is a variable of $\KK[y, z]$. 
One can prove that a derivation $J(\,\cdot\,, u)$ is homogeneous if and only if $u$ is a sum of a homogeneous element of $\KK[y, z]$ and a constant from~$\KK$. 
Since the constant does not affect $J(\,\cdot\,, u)$, we can suppose that $g(v)$ is a homogeneous element of $\KK[y, z]$ with respect to the grading~\eqref{tab}. 
It is easily shown that $g(v)$ is homogeneous as an element of $\KK[y, z]$ if and only if $g$ is homogeneous as a polynomial and $v$ is homogeneous as an element of $\KK[y, z]$. 
Hence $g(v) = \eta v^m$ for some $\eta \in \KK$, $m \in \NN$, and $v \in \KK[y, z]$ is homogeneous with respect to~\eqref{tab}. 

On the other side, $v$ is irreducible as a polynomial in $y$ and $z$ since it is a variable of $\KK[y, z]$. 
It can be proved that irreducible polynomials in $\KK[y, z]$, which are homogeneous with respect to~\eqref{tab}, have the form $\zeta y + \xi z^k$, $\zeta, \xi \in \KK$. 
Thus we get: 
$$\rpa(\,\cdot\,) = J(\,\cdot\,, g(v)) = J(\,\cdot\,, \eta v^m) = J(\,\cdot\,, \eta(\zeta y + \xi z^k)^m),$$
where $a, b, c \in \KK$, $m \in \NN$. 

Let us calculate $\rpa(y)$ and $\rpa(z^k)$. 
Note that for any $f \in \KK[y, z]$
$$\rpa(f) = J(f, \eta v^m) = \eta mv^{m-1} J(f, v) = 
\eta mv^{m-1} J(f, \zeta y + \xi z^k) = \eta mv^{m-1} \det
\begin{pmatrix}
f_y & f_z \\
\zeta & \xi kz^{k-1}
\end{pmatrix},$$
whence we get $\rpa(y) = \eta \xi mkv^{m-1}z^{k-1}$ and 
$\rpa(z^k) = -\eta \zeta mkv^{m-1}z^{k-1}$. 
Thus we prove that $\rpa(y)$ and $\rpa(z^k)$ are proportional over $\KK$ in $\KK[y, z]$. Since $\pa(y) = \rpa(y)$, $\pa(z) = \rpa(z)$, $\pa(x) \hm= -\rpa(y)-\rpa(z^k)$, we conclude that $\pa(x)$, $\pa(y)$, and $\pa(z^k)$ are proportional in~$\R$. 
\end{proof}

Now let us reduce Proposition~\ref{pr_23m} to the previous lemma. 

\begin{proof}[Proof of Proposition~\ref{pr_23m}]
Assume that there exist at most two non-kernel variables. 
Then at least one of $\delta(T_0^{l_0})$, $\delta(T_1^{l_1})$, $\delta(T_2^{l_2})$ is equal to $0$, 
and $\delta(\g) = \delta(0) = 0$ implies the assertion.
Thus without loss of generality it can be assumed that there are three non-kernel variables $T_{01}$, $T_{11}$, and $T_{21}$ with non-kernel exponents $l_{01} = l_{11} = 1$, $l_{21} = k$. 

\smallskip

Let $\LL$ be the algebraic closure of the field obtained from $\KK$ by adding all variables from the kernel of $\delta$:
$$\LL = \overline{\KK(T_{02}, \ldots, T_{0n_0}, T_{12}, \ldots, T_{1n_1}, T_{22}, \ldots, T_{2n_2})}.$$

Consider $\{T_0^{l_0} + T_1^{l_1} + T_2^{l_2} = 0\}$ as an algebraic variety over $\LL$ in $3$-dimensional space with coordinates $T_{01}$, $T_{11}$, and $T_{21}$. Denote $\LL[T_{i1}] := \LL[T_{01}, T_{11}, T_{21}]$. Since $\delta(T_{ij}) = 0$ for $j \ne 1$, we can consider $\delta$ as a derivation of $\LL[T_{i1}]\,/\,(\g)$. 
Indeed, a map $\delta$ defined by the same 
$\delta(T_{ij}) \hm \in \KK[T_{ij}]\,/\,(\g) \hm \subseteq \LL[T_{i1}]\,/\,(\g)$ 
and extended to a map 
$\LL[T_{i1}] \to \LL[T_{i1}]\,/\,(\g)$ by Leibniz rule, induces a well-defined derivation 
$\LL[T_{i1}]\,/\,(\g) \to \LL[T_{i1}]\,/\,(\g)$ since $\delta(\g) = 0$. 
Clearly, this $\delta$ is also locally nilpotent. 

\begin{lemma}
\label{lemma1}
$\delta(T_{0}^{l_0})$, $\delta(T_{1}^{l_1})$, and $\delta(T_{2}^{l_2})$ are proportional over $\LL$. 
\end{lemma}

\begin{proof}
After a linear coordinate transformation
\begin{gather*}
x = T_{01} \cdot T_{02}^{l_{02}} \ldots T_{0n_0}^{l_{0n_0}} = T_0^{l_0}\\
y = T_{11} \cdot T_{12}^{l_{12}} \ldots T_{1n_1}^{l_{1n_1}} = T_1^{l_1}\\
z = T_{21} \cdot \sqrt[k]{T_{22}^{l_{22}} \ldots T_{2n_2}^{l_{2n_2}}} = \sqrt[k]{T_2^{l_2}}
\end{gather*}
we obtain the hypersurface $\{x + y + z^k = 0\} \subseteq \LL^3$ and the corresponding locally nilpotent derivation $\pa$. 
Notice that
$$
\rdeg T_{01} = k, \quad \rdeg T_{11} = k, \quad \rdeg T_{21} = 1, \quad \rdeg T_{ij} = 0 \;\text{for all other $(i, j)$}
$$
is a well-defined $\ZZ$-grading on $\KK[T_{ij}]\,/\,(\g)$ since $\g$ is homogeneous (of degree $k$). 
By Lemma~\ref{l_rdeg}, $\delta$ is homogeneous with respect to this grading, 
whence $\pa$ is homogeneous with respect to
$$
\rdeg x = k, \quad \rdeg y = k, \quad \rdeg z = 1.
$$
Applying Lemma~\ref{l_base} to the derivation $\pa$ and the ground field $\LL$, we obtain that $\pa(x)$, $\pa(y)$, and $\pa(z^k)$ lie in a subspace of~$\LL[x, y, z]\,/\,(x+y+z^k)$ of dimension~$1$. 
For the derivation~$\delta$ it means that $\delta(T_0^{l_0})$, $\delta(T_1^{l_1})$, and $\delta(T_2^{l_2})$ lie in a subspace of~$\LL[T_{i1}]\,/\,(\g)$ of dimension~$1$. 
\end{proof}

Fix any $s, t \in \{0, 1, 2\}$, $s \ne t$. According to Lemma~\ref{lemma1}, there exists $\ell \in \LL$ such that
\begin{equation}
\label{lprop}
\delta(T_s^{l_s}) = \ell\, \delta(T_t^{l_t}) \;\text{ in } \LL[T_{i1}]\,/\,(\g).
\end{equation}
Then
$$
\label{rat_l}
\frac{\delta(T_s^{l_s})}{\delta(T_t^{l_t})} = \ell \;\text{ in }
Q\left(\LL[T_{i1}]\,/\,(\g)\right).
$$
On the other hand, $\deg \delta(T_s^{l_s}) = \deg \delta(T_t^{l_t})$. By Lemma~\ref{fracfield} it follows that
$$
\label{rat_frf}
\frac{\delta(T_s^{l_s})}{\delta(T_t^{l_t})} \in 
\KK\left(\frac{T_1^{l_1}}{T_2^{l_2}}\right) \;\text{ in }
Q\left(\KK[T_{ij}]\,/\,(\g)\right)_0 
\subseteq
Q\left(\LL[T_{i1}]\,/\,(\g)\right).
$$
Every element of $\KK\left(\dfrac{u}{v}\right)$ can be written in the form $\dfrac{f_1(u, v)}{f_2(u, v)}$, where $f_1, f_2 \in \KK[u, v]$. Whence 
\begin{align*}
\frac{f_1(T_1^{l_1}, T_2^{l_2})}{f_2(T_1^{l_1}, T_2^{l_2})} &= \ell 
&&\text{in }
Q\left(\LL[T_{i1}]\,/\,(\g)\right);
\\
f_1(T_1^{l_1}, T_2^{l_2}) - \ell f_2(T_1^{l_1}, T_2^{l_2}) &= 0
&&\text{in }
\LL[T_{i1}]\,/\,(\g);
\\
f_1(T_1^{l_1}, T_2^{l_2}) - \ell f_2(T_1^{l_1}, T_2^{l_2}) &= \g \cdot h
&&\text{in } \LL[T_{i1}] \text{ for some } h \in \LL[T_{i1}].
\\
\intertext{Since the left side does not depend on $T_{01}$,}
f_1(T_1^{l_1}, T_2^{l_2}) - \ell f_2(T_1^{l_1}, T_2^{l_2}) &= 0
&&\text{in } \LL[T_{i1}],
\end{align*}
which follows that $f_1$ and $f_2$ are proportional with coefficient $\ell$ as polynomials in two variables. 
At the same time $f_1, f_2$ are polynomials over $\KK$. 
Thus we get that $\ell \in \KK$ in~\eqref{lprop} and~\eqref{lprop} holds in $\KK[T_{ij}]\,/\,(\g)$. This concludes the proof of Proposition~\ref{pr_23m}.
\end{proof}

\section{Roots of trinomial algebras}
\label{sect_cor}

Let a quasitorus $H$ act on an algebraic variety~$X$. Consider the corresponding grading on~$\KK[X]$. 
Recall that a \textit{root} of~$X$ is the degree of a homogeneous locally nilpotent derivation of~$\KK[X]$. 

\begin{proposition}
\label{pr_roots}
Let $\R = \KK[\X]$ be a trinomial algebra.  
For any $a, b \in \{0, 1, 2\}$, $a \ne b$, $1 \le c_a \le n_a$, and $1 \le c_b \le n_b$ such that $l_{bc_b} = 1$, denote 
$$
E(T_{ac_a}, T_{bc_b}) = \biggl\{
\deg \g - \deg T_{ac_a} - \deg T_{bc_b} + \!\!\!
\sum \limits_{\substack{(i, j) \ne (a, c_a) \\ (i, j) \ne (b, c_b)}} 
\!\!\! u_{ij} \deg T_{ij} \;
\biggm| \; u_{ij} \in \Zgezero
\biggr\} \subseteq K.
$$
Then $e \in K$ is a root of~$\X$ if and only if $e \in \bigcup E(T_{ac_a}, T_{bc_b})$. 
\end{proposition}

\begin{proof}
According to Theorem~\ref{theor}, we have to find the degrees of elementary derivations. Note that 
$$
\deg \dCb = 
\begin{cases}
2 \deg \g - \sum \limits_{i} \deg T_{ic_i} 
& \text{in case (i) of Construction~\ref{dCb_constr}}\\
\deg \g - \sum \limits_{i \ne i_0} \deg T_{ic_i}
 & \text{in case (ii) of Construction~\ref{dCb_constr}}
\end{cases}
$$
Every elementary derivation $\delta$ has the form $h \dCb$, where $h$ is a homogeneous element of~$\Ker \dCb$. 
By Proposition~\ref{prop_ker}, 
$\deg h = \!\!\! \sum \limits_{T_{ij} \in \Ker \dCb} \!\!\! u_{ij} \deg T_{ij} + m \deg \g$ for some $u_{ij}, m \in \Zgezero$. Hence
\begin{equation}
\label{degdelta}
\deg \delta = 
\begin{cases}
2 \deg \g - \sum \limits_{i} \deg T_{ic_i} +
\!\!\! \sum \limits_{T_{ij} \in \Ker \dCb} \!\!\! u_{ij} \deg T_{ij} + m \deg \g 
& \text{in case (i)}\\
\deg \g - \sum \limits_{i \ne i_0} \deg T_{ic_i} +
\!\!\! \sum \limits_{T_{ij} \in \Ker \dCb} \!\!\! u_{ij} \deg T_{ij} + m \deg \g
 & \text{in case (ii)}
\end{cases}
\end{equation}
We claim that for any elementary $\delta = h \dCb$ of Type~I there is $\delta\pr = h\pr \delta_{C\pr, \beta\pr}$ of Type~II with $\deg \delta = \deg \delta\pr$. 
Indeed, for any $C = (c_0, c_1, c_2)$, $\beta$ satisfying the conditions of case~(i), 
take any $i_0 \in \{0, 1, 2\}$ and $\beta\pr = (\beta_0\pr, \beta_1\pr, \beta_2\pr) \ne 0$ with $\beta_{i_0}\pr = 0$, $\beta_0\pr + \beta_1\pr + \beta_2\pr = 0$. 
Then the data $C\pr = C$, $\beta\pr$ satisfies the conditions of case~(ii), and for $\delta_1 = T_{i_0}^{l_{i_0}} / T_{i_0c_{i_0}} \cdot \delta_{C\pr, \beta\pr}$ we have 
$$
\deg \delta_1
= \deg \g - \deg T_{i_0c_{i_0}} + \deg \g - \sum \limits_{i \ne i_0} \deg T_{ic_i} 
= 2 \deg \g - \sum \limits_{i} \deg T_{ic_i} = \deg \dCb.
$$
Note that $T_{i_0}^{l_{i_0}} / T_{i_0c_{i_0}}$, 
$T_{i_0}^{l_{i_0}}$, and $T_{ij} \in \Ker \dCb$ belong to 
$\Ker \delta_{C\pr, \beta\pr}$. 
Replacing the factor of degree $m \deg \g$ in $h$ by $\left(T_{i_0}^{l_{i_0}}\right)^m$, we get the element $h_1 \in \Ker \delta_{C\pr, \beta\pr}$ of degree $\deg h$. 
Then $\delta\pr = h_1 \delta_1$ is elementary of Type~II and $\deg \delta\pr = \deg \delta$, as required. 

It remains to note that $\deg \g = \sum \limits_{j} l_{i_0j} \deg T_{i_0j}$ and set $\{a, b\} = \{0, 1, 2\} \setminus \{i_0\}$. 
\end{proof}

We call a set of the form $E(T_{ac_a}, T_{bc_b})$ a \textit{basic set}. 

\begin{corollary}
\label{cor_derex}
For any $e \in E(T_{ac_a}, T_{bc_b})$ there exists locally nilpotent derivation $\delta$ with non-kernel variables $T_{ac_a}$, $T_{bc_b}$ and $\deg \delta = e$. 
\end{corollary}

\begin{proof}
The statement follows from~\eqref{degdelta} and Proposition~\ref{prop_ker}. 
\end{proof}

The following lemma states that the degree of a locally nilpotent derivation determines a unique variable in every monomial~$T_i^{l_i}$ that can be non-kernel. 

\begin{lemma}
\label{lem_1root}
Let $\deg \delta = \deg \delta'$, $T_{ic} \notin \Ker \delta$, $T_{ic'} \notin \Ker \delta'$. Then $c = c'$. 
\end{lemma}

\begin{proof}
Suppose that $c \ne c'$. Without loss of generality we can assume that $c = 1$, $c' = 2$, $i = 0$. 
Consider the following $\ZZ$-grading on the algebra~$\KK[T_{ij}]$:
$$
\rdeg T_{01} = l_{02}, \quad \rdeg T_{02} = -l_{01}, \quad \rdeg T_{ij} = 0 \;\text{ for all other $(i, j)$.}
$$
The trinomial $\g$ is homogeneous (of degree $0$) with respect to this grading, therefore the \hbox{$\rdeg$-grading} is a well-defined grading on the factor algebra~$\R$. 
By Lemma~\ref{l_rdeg}, it follows that $\delta$ and $\delta'$ are $\rdeg$-homogeneous and $\rdeg \delta = \rdeg \delta'$. 

Let $\rdeg \delta \ge 0$, then $\rdeg \delta(T_{01}) > 0$. 
Note that $T_{01}$ is a unique variable with a positive degree. 
Hence every monomial in $\delta(T_{01})$ includes $T_{01}$ and therefore $T_{01}$ divides $\delta(T_{01})$. But $\delta(T_{01}) \ne 0$ since $T_{01}$ is non-kernel. This contradicts Lemma~\ref{fr_lem}\ref{fr_lem_f}. 

Thus, $\rdeg \delta < 0$.
By the same reason, $\rdeg \delta' > 0$, but they are equal, a contradiction. 
\end{proof}

\begin{proposition}
\label{pr_3roots}
Every $e \in K$ belongs to at most three basic sets. 
\end{proposition}

\begin{proof}
Let $e \in E(T_{ac_a}, T_{bc_b})$. 
According to Corollary~\ref{cor_derex}, there exists locally nilpotent derivation $\delta$ with $\deg \delta = e$ and $T_{ac_a}, T_{bc_b} \notin \Ker \delta$. 
By Lemma~\ref{lem_1root} there is at most one variable $T_{ic_i}$ in every monomial $T_i^{l_i}$ that can be non-kernel for some locally nilpotent derivation with the degree $e$. 
Then $e$ can belong only to $E(T_{0c_0}, T_{1c_1}), E(T_{0c_0}, T_{2c_2}), E(T_{1c_1}, T_{2c_2})$.
\end{proof}

\begin{remark}
\label{rem_TI}
It follows from the proof of Proposition~\ref{pr_roots} that any degree of a derivation of Type~I belongs to all three possible basic sets. The converse is false, see Example~\ref{ex_pyu1}.
\end{remark}

In \cite{Ko2014_eng}, homogeneous locally nilpotent derivations for quadrics with a complexity one torus action are described. 
This description is obtained by applying the technique of proper polyhedral divisors (see~\cite{Li12010} and~\cite{Li22010}). 
We obtain the following statements from Theorem~\ref{theor} and Proposition~\ref{prop_ker} (Examples~\ref{ex_pyu2},~\ref{ex_pyu1}).

\begin{example}
\label{ex_pyu2}
Every finely homogeneous locally nilpotent derivation of the algebra~$\R$, where 
$\g = T_{01}T_{02} + T_{11}T_{12} + T_{21}T_{22}$, has the form 
\begin{gather*}
T_{0i_0}^{k_0} T_{1i_1}^{k_1} T_{2i_2}^{k_2} (\alpha_2 T_{11}T_{12} - \alpha_1 T_{21}T_{22})^p 
\left(\alpha_0 T_{1i_1} T_{2i_2} \dfrac{\partial}{\partial T_{0\bar i_0}} 
+ \alpha_1 T_{0i_0} T_{2i_2} \dfrac{\partial}{\partial T_{1\bar i_1}} 
+ \alpha_2 T_{0i_0} T_{1i_1} \dfrac{\partial}{\partial T_{2\bar i_2}} 
\right), 
\end{gather*}
where $k_0, k_1, k_2, p \in \Zgezero$, $\{i_0, \bar i_0\} = \{i_1, \bar i_1\} = \{i_2, \bar i_2\} = \{1, 2\}$,  
$\alpha_0, \alpha_1, \alpha_2 \in \KK$, and \hbox{$\alpha_0 + \alpha_1 + \alpha_2 = 0$}. 
\end{example}

\begin{example}
\label{ex_pyu1}
Every finely homogeneous locally nilpotent derivation of the algebra~$\R$, where 
$\g = T_{01}T_{02} + T_{11}T_{12} + T_{21}^2$, has the form 
\begin{gather*}
\lambda T_{0i}^k T_{1j}^l T_{21}^p 
\left(T_{1j} \dfrac{\partial}{\partial T_{1\bar i}} - T_{0i} \dfrac{\partial}{\partial T_{1\bar j}}\right) \quad\text{or} 
\\
T_{0i}^k T_{1j}^l (\alpha_1 T_{01}T_{02} - \alpha_0 T_{11}T_{12})^p 
\left(\alpha_0 T_{1j} T_{21} \dfrac{\partial}{\partial T_{0\bar i}} 
+ \alpha_1 T_{0i} T_{21} \dfrac{\partial}{\partial T_{1\bar j}}
- \dfrac{\alpha_0 + \alpha_1}{2} T_{0i} T_{1j} \dfrac{\partial}{\partial T_{21}}
\right), 
\end{gather*}
where $k, l, p \in \Zgezero$, $\{i, \bar i\} = \{j, \bar j\} = \{1, 2\}$,  $\alpha_0, \alpha_1, \lambda \in \KK$. 

\begin{remark}
There is no factor $(T_{01}T_{02} + T_{11} T_{12})^k$ in the first case since it can be written as $(-1)^k T_{21}^{2k}$. 
\end{remark}
\begin{figure}[ht]
    \centering
	\newcounter{x}
	\newcounter{y}
    \newcounter{z}
    \newcounter{xy}
    \newcounter{xyz}

\begin{minipage}{.5\textwidth} 
\centering
\begin{tikzpicture}
    \pgftransformcm{0.4}{0}{0}{0.4}{\pgfpoint{0cm}{0cm}}
    \tikzstyle{conefill} = [fill=blue!20,fill opacity=0.7, draw = black!70]

    \coordinate (Btwo) at (1.4,0);
    \coordinate (Bthr) at (0,1.4);
    \coordinate (Bone) at (-1.2,-0.6);

    \coordinate (YAxisMin) at ($-4*(Btwo)$);
    \coordinate (YAxisMax) at ($4*(Btwo)$);
    \coordinate (ZAxisMin) at ($-3*(Bthr)$);
    \coordinate (ZAxisMax) at ($7*(Bthr)$);
    \coordinate (XAxisMin) at ($-5*(Bone)$);
    \coordinate (XAxisMax) at ($5*(Bone)$);

    \draw [thick, black,-latex] (YAxisMin) -- (YAxisMax) node [below] {$y$};
    \draw [thick, black,-latex] (XAxisMin) -- (XAxisMax) node [below right] {$x$};
    \draw [thick, black,-latex] (ZAxisMin) -- (ZAxisMax) node [below right] {$z$};
    \draw [-latex] (0, 0) -- (Btwo) -- ($(Bone)+(Btwo)$) -- (Bone) -- cycle;

\filldraw[conefill] ($(Bone)+(Btwo)$) -- ($(Bone)+(Btwo)+4*(Bthr)$) -- ($5*(Bone)+(Btwo)+4*(Bthr)$) -- cycle;
\filldraw[conefill] ($(Bone)+(Btwo)$) -- ($(Bone)+(Btwo)+4*(Bthr)$) -- ($(Bone)+5*(Btwo)+4*(Bthr)$) -- cycle;

  \foreach \z in {1,...,6}{
    \forloop{xy}{1}{\value{xy} < \z}{
      \forloop{y}{1}{\value{y} < \value{xy}}{
       \node[draw,circle,blue!40, inner sep=1.5pt,fill] at ($\value{xy}*(Bone)-\value{y}*(Bone)+\value{y}*(Btwo)+\z*(Bthr)-2*(Bthr)$) {}; 
      }
    }
  }

\filldraw[conefill] ($(Bone)+(Btwo)$) -- ($5*(Bone)+(Btwo)+4*(Bthr)$) -- ($(Bone)+5*(Btwo)+4*(Bthr)$) -- cycle;

\foreach \x in {1,...,6}{
  \forloop{xy}{\x}{\value{xy} < 6}{
    \node[draw,circle, blue!40, inner sep=1.5pt,fill] at ($\x*(Bone)+\value{xy}*(Btwo)-\x*(Btwo)+(Btwo)+\value{xy}*(Bthr)-(Bthr)$) {}; 
  }
}
\label{cone_fig}[1]
\end{tikzpicture}
\caption{A~<<cone>> $E(T_{01}, T_{11})$}
\end{minipage}%
\centering
\begin{minipage}{.5\textwidth} 
\centering
\begin{tikzpicture}
    \pgftransformcm{0.4}{0}{0}{0.4}{\pgfpoint{0cm}{0cm}}
    \tikzstyle{conefill} = [fill=red!20,fill opacity=0.7, draw = black!70]

    \coordinate (Btwo) at (1.4,0);
    \coordinate (Bthr) at (0,1.4);
    \coordinate (Bone) at (-1.2,-0.6);
    \coordinate (Or) at (0,0);
    \coordinate (YAxisMin) at ($-4*(Btwo)$);
    \coordinate (YAxisMax) at ($4*(Btwo)$);
    \coordinate (ZAxisMin) at ($-3*(Bthr)$);
    \coordinate (ZAxisMax) at ($7*(Bthr)$);
    \coordinate (XAxisMin) at ($-5*(Bone)$);
    \coordinate (XAxisMax) at ($5*(Bone)$);

    \draw [thick, black,-latex] (YAxisMin) -- (YAxisMax) node [below] {$y$};
    \draw [thick, black,-latex] (XAxisMin) -- (XAxisMax) node [below right] {$x$};
    \draw [thick, black,-latex] (ZAxisMin) -- (ZAxisMax) node [below right] {$z$};
    \draw [-latex] ($-1*(Bone)-(Btwo)$) -- ($-1*(Bone)+(Btwo)$) -- ($(Bone)+(Btwo)$) -- ($(Bone)-(Btwo)$) -- cycle;

\filldraw[conefill] ($(Bone)$) -- ($(Bone)+4*(Btwo)+4*(Bthr)$) -- ($(Bone)-4*(Btwo)+4*(Bthr)$) -- cycle;
\filldraw[conefill] ($(Bone)$) -- ($(Bone)-4*(Btwo)+4*(Bthr)$) -- ($5*(Bone)+4*(Bthr)$) -- cycle;

  \foreach \xyz in {0,1,...,6}{
    \forloop{xy}{0}{\value{xy} < \xyz}{
      \forloop{x}{0}{\value{x} < \value{xy}}{
       \node[draw,circle,red, inner sep=0.5pt,fill] at 
($5*(Bthr)
+\value{x}*(Bone)+(Bone)
+\value{xy}*(Btwo)-\value{x}*(Btwo)-(Btwo)
-\xyz*(Bthr)+\value{xy}*(Bthr)$) {}; 
             \node[draw,circle,red, inner sep=0.5pt,fill] at 
($5*(Bthr)
+\value{x}*(Bone)+(Bone)
-\value{xy}*(Btwo)+\value{x}*(Btwo)+(Btwo)
-\xyz*(Bthr)+\value{xy}*(Bthr)$) {}; 

      }
    }
  }

  \foreach \xyz in {0,2,...,6}{
    \forloop{xy}{0}{\value{xy} < \xyz}{
      \forloop{x}{0}{\value{x} < \value{xy}}{
       \node[draw,circle,red, inner sep=1.5pt,fill] at 
($5*(Bthr)
+\value{x}*(Bone)+(Bone)
+\value{xy}*(Btwo)-\value{x}*(Btwo)-(Btwo)
-\xyz*(Bthr)+\value{xy}*(Bthr)$) {}; 
             \node[draw,circle,red, inner sep=1.5pt,fill] at 
($5*(Bthr)
+\value{x}*(Bone)+(Bone)
-\value{xy}*(Btwo)+\value{x}*(Btwo)+(Btwo)
-\xyz*(Bthr)+\value{xy}*(Bthr)$) {}; 
      }
    }
  }

\filldraw[conefill] ($(Bone)$) -- ($(Bone)+4*(Btwo)+4*(Bthr)$) -- ($5*(Bone)+4*(Bthr)$) -- cycle;

\foreach \x in {1,...,6}{
  \forloop{xy}{\x}{\value{xy} < 6}{
    \node[draw,circle, red, inner sep=1.5pt,fill] at ($\x*(Bone)+\value{xy}*(Btwo)-\x*(Btwo)+\value{xy}*(Bthr)-(Bthr)$) {}; 
  }
}
\end{tikzpicture}
\caption{A~<<cone>> $E(T_{01}, T_{21})$}
\end{minipage}
\begin{minipage}{.28\textwidth} 
\centering
\begin{tikzpicture}
    \pgftransformcm{0.4}{0}{0}{0.4}{\pgfpoint{0cm}{0cm}}
    \tikzstyle{conefill} = [fill=blue!20,fill opacity=0.7, draw = black!70]

    \coordinate (Btwo) at (1.4,0);
    \coordinate (Bthr) at (0,1.4);
    \coordinate (Bone) at (-1.2,-0.6);
    \coordinate (Or) at (0,0);
    \coordinate (YAxisMin) at ($-3*(Btwo)$);
    \coordinate (YAxisMax) at ($3*(Btwo)$);
    \coordinate (ZAxisMin) at ($-3*(Bthr)$);
    \coordinate (ZAxisMax) at ($7*(Bthr)$);
    \coordinate (XAxisMin) at ($-3*(Bone)$);
    \coordinate (XAxisMax) at ($5*(Bone)$);

    \draw [thick, black,-latex] (YAxisMin) -- (YAxisMax) node [below] {$y$};
    \draw [thick, black,-latex] (XAxisMin) -- (XAxisMax) node [below right] {$x$};
    \draw [thick, black,-latex] (ZAxisMin) -- (ZAxisMax) node [below right] {$z$};
    \draw [-latex] ($-1*(Bone)-(Btwo)$) -- ($-1*(Bone)+(Btwo)$) -- ($(Bone)+(Btwo)$) -- ($(Bone)-(Btwo)$) -- cycle;

\draw[red!40,-latex, -] (Bone) -- ($(Bone)+4*(Bthr)$) -- ($5*(Bone)+4*(Bthr)$) -- cycle;
\foreach \z in {1,...,6}{
  \forloop{x}{1}{\value{x} < \z}{
    \node[draw,circle,red, inner sep=0.5pt,fill] at ($\value{x}*(Bone)+\z*(Bthr)-2*(Bthr)$) {}; 
  }
}
\foreach \x in {1,...,6}{
 \forloop[2]{z}{\x}{\value{z} < 6}{
        \node[draw,circle,red, inner sep=1.5pt,fill] at ($\x*(Bone)+\value{z}*(Bthr)-(Bthr)$) {};
  }
}
\end{tikzpicture}
\label{layer_fig}
\caption{A~<<layer>>}
\end{minipage}%
\begin{minipage}{.44\textwidth} 
\centering
\begin{tikzpicture}
    \pgftransformcm{0.4}{0}{0}{0.4}{\pgfpoint{0cm}{0cm}}
    \tikzstyle{conefill} = [fill=blue!20,fill opacity=0.7, draw = black!70]

    \coordinate (Btwo) at (1.4,0);
    \coordinate (Bthr) at (0,1.4);
    \coordinate (Bone) at (-1.2,-0.6);
    
    \coordinate (Or) at (0,0);
    \coordinate (YAxisMin) at ($-5*(Btwo)$);
    \coordinate (YAxisMax) at ($5*(Btwo)$);
    \coordinate (ZAxisMin) at ($-3*(Bthr)$);
    \coordinate (ZAxisMax) at ($7*(Bthr)$);
    \coordinate (XAxisMin) at ($-5*(Bone)$);
    \coordinate (XAxisMax) at ($5*(Bone)$);

    \draw [thick, black,-latex] (YAxisMin) -- (YAxisMax) node [below] {$y$};
    \draw [thick, black,-latex] (XAxisMin) -- (XAxisMax) node [below right] {$x$};
    \draw [-latex] ($-1*(Bone)-(Btwo)$) -- ($-1*(Bone)+(Btwo)$) -- ($(Bone)+(Btwo)$) -- ($(Bone)-(Btwo)$) -- cycle;

\filldraw[conefill] ($-1*(Bone)-(Btwo)$) -- ($-5*(Bone)-(Btwo)+4*(Bthr)$) -- ($-1*(Bone)-5*(Btwo)+4*(Bthr)$) -- cycle;
\filldraw[conefill] ($-1*(Bone)-(Btwo)$) --  ($-1*(Bone)-(Btwo)+4*(Bthr)$) -- ($-5*(Bone)-(Btwo)+4*(Bthr)$) -- cycle;
\filldraw[conefill] ($-1*(Bone)-(Btwo)$) -- ($-1*(Bone)-(Btwo)+4*(Bthr)$) -- ($-1*(Bone)-5*(Btwo)+4*(Bthr)$) -- cycle;

\draw[red!40,-latex, -] ($-1*(Bone)$) -- ($-1*(Bone)+4*(Bthr)$) -- ($-5*(Bone)+4*(Bthr)$) -- cycle;
\foreach \x in {1,...,6}{
 \forloop[2]{z}{\x}{\value{z} < 6}{
        \node[draw,circle,red, inner sep=1.5pt,fill] at ($-\x*(Bone)+\value{z}*(Bthr)-(Bthr)$) {};
  }
}

\filldraw[conefill] ($-1*(Bone)+(Btwo)$) -- ($-5*(Bone)+(Btwo)+4*(Bthr)$) -- ($-1*(Bone)+5*(Btwo)+4*(Bthr)$) -- cycle;
\filldraw[conefill] ($-1*(Bone)+(Btwo)$) --  ($-1*(Bone)+(Btwo)+4*(Bthr)$) -- ($-5*(Bone)+(Btwo)+4*(Bthr)$) -- cycle;
\filldraw[conefill] ($-1*(Bone)+(Btwo)$) -- ($-1*(Bone)+(Btwo)+4*(Bthr)$) -- ($-1*(Bone)+5*(Btwo)+4*(Bthr)$) -- cycle;

\draw[teal!60,-latex, -] (Btwo) -- ($(Btwo)+4*(Bthr)$) -- ($5*(Btwo)+4*(Bthr)$) -- cycle;
\foreach \y in {1,...,6}{
 \forloop[2]{z}{\y}{\value{z} < 6}{
        \node[draw,circle,teal, inner sep=1.5pt,fill] at ($\y*(Btwo)+\value{z}*(Bthr)-(Bthr)$) {};
  }
}

\draw[teal!60,-latex, -] ($-1*(Btwo)$) -- ($-1*(Btwo)+4*(Bthr)$) -- ($-5*(Btwo)+4*(Bthr)$) -- cycle;
\foreach \y in {1,...,6}{
 \forloop[2]{z}{\y}{\value{z} < 6}{
        \node[draw,circle,teal, inner sep=1.5pt,fill] at ($-\y*(Btwo)+\value{z}*(Bthr)-(Bthr)$) {};
  }
}

    \draw [thick, black,-latex] (ZAxisMin) -- (ZAxisMax) node [below right] {$z$};

\filldraw[conefill] ($(Bone)-(Btwo)$) -- ($(Bone)-(Btwo)+4*(Bthr)$) -- ($(Bone)-5*(Btwo)+4*(Bthr)$) -- cycle;
\filldraw[conefill] ($(Bone)-(Btwo)$) -- ($5*(Bone)-(Btwo)+4*(Bthr)$) -- ($(Bone)-5*(Btwo)+4*(Bthr)$) -- cycle;
\filldraw[conefill] ($(Bone)-(Btwo)$) -- ($(Bone)-(Btwo)+4*(Bthr)$) -- ($5*(Bone)-(Btwo)+4*(Bthr)$) -- cycle;

\draw[red!40,-latex, -] (Bone) -- ($(Bone)+4*(Bthr)$) -- ($5*(Bone)+4*(Bthr)$) -- cycle;
\foreach \x in {1,...,6}{
 \forloop[2]{z}{\x}{\value{z} < 6}{
        \node[draw,circle,red, inner sep=1.5pt,fill] at ($\x*(Bone)+\value{z}*(Bthr)-(Bthr)$) {};
  }
}

\filldraw[conefill] ($(Bone)+(Btwo)$) -- ($(Bone)+(Btwo)+4*(Bthr)$) -- ($5*(Bone)+(Btwo)+4*(Bthr)$) -- cycle;
\filldraw[conefill] ($(Bone)+(Btwo)$) -- ($(Bone)+(Btwo)+4*(Bthr)$) -- ($(Bone)+5*(Btwo)+4*(Bthr)$) -- cycle;
\filldraw[conefill] ($(Bone)+(Btwo)$) -- ($5*(Bone)+(Btwo)+4*(Bthr)$) -- ($(Bone)+5*(Btwo)+4*(Bthr)$) -- cycle;
\end{tikzpicture}
\label{roots_fig}
\caption{All~roots}
\end{minipage}%
\begin{minipage}{.28\textwidth} 
\centering
\begin{tikzpicture}
    \pgftransformcm{0.4}{0}{0}{0.4}{\pgfpoint{0cm}{0cm}}
    \tikzstyle{conefill} = [fill=blue!20,fill opacity=0.7, draw = black!70]
    \tikzstyle{conefill2} = [fill=violet!40,fill opacity=0.7, draw = black!70]

    \coordinate (Btwo) at (1.4,0);
    \coordinate (Bthr) at (0,1.4);
    \coordinate (Bone) at (-1.2,-0.6);
    
    \coordinate (Or) at (0,0);
    \coordinate (YAxisMin) at ($-3*(Btwo)$);
    \coordinate (YAxisMax) at ($3*(Btwo)$);
    \coordinate (ZAxisMin) at ($-3*(Bthr)$);
    \coordinate (ZAxisMax) at ($7*(Bthr)$);
    \coordinate (XAxisMin) at ($-3*(Bone)$);
    \coordinate (XAxisMax) at ($5*(Bone)$);

    \draw [thick, black,-latex] (YAxisMin) -- (YAxisMax) node [below] {$y$};
    \draw [thick, black,-latex] (XAxisMin) -- (XAxisMax) node [below right] {$x$};
    \draw [-latex] ($-1*(Bone)-(Btwo)$) -- ($-1*(Bone)+(Btwo)$) -- ($(Bone)+(Btwo)$) -- ($(Bone)-(Btwo)$) -- cycle;
    \draw [thick, black,-latex] (ZAxisMin) -- (ZAxisMax) node [below right] {$z$};

\filldraw[conefill] ($(Bone)+(Btwo)$) -- ($(Bone)+(Btwo)+4*(Bthr)$) -- ($5*(Bone)+(Btwo)+4*(Bthr)$) -- cycle;
\filldraw[conefill] ($(Bone)+(Btwo)$) -- ($(Bone)+(Btwo)+4*(Bthr)$) -- ($(Bone)+5*(Btwo)+4*(Bthr)$) -- cycle;
\filldraw[conefill2] ($(Bone)+(Btwo)+(Bthr)$) -- ($4*(Bone)+(Btwo)+4*(Bthr)$) -- ($(Bone)+4*(Btwo)+4*(Bthr)$) -- cycle;
\filldraw[conefill] ($(Bone)+(Btwo)$) -- ($5*(Bone)+(Btwo)+4*(Bthr)$) -- ($(Bone)+5*(Btwo)+4*(Bthr)$) -- cycle;
\end{tikzpicture}
\caption{Type~I}
\end{minipage}

\end{figure}

Let us describe the roots in Example~\ref{ex_pyu1}. The grading group $K$ is isomorphic to $\ZZ^3$, and the grading can be given explicitly via
$$
\deg T_{01} = \begin{pmatrix}1 \\ 0 \\ 1\end{pmatrix}, \;
\deg T_{02} = \begin{pmatrix}-1 \\ 0 \\ 1\end{pmatrix}, \;
\deg T_{11} = \begin{pmatrix}0 \\ 1 \\ 1\end{pmatrix}, \;
\deg T_{12} = \begin{pmatrix}0 \\ -1 \\ 1\end{pmatrix}, \;
\deg T_{21} = \begin{pmatrix}0 \\ 0 \\ 1\end{pmatrix}.
$$
By applying Proposition~\ref{pr_roots}, one can find the set of roots of the quadric. Namely, it is the union of eight basic sets. 
Four <<corner>> basic sets of the form $E(T_{0c_0}, T_{1c_1})$ consist of all integer points in cones with vertices at $(\pm1, \pm 1, 0)$, one of them is shown in Figure~1. 
Four <<lateral>> basic sets of the form $E(T_{0c_0}, T_{21})$ and $E(T_{1c_1}, T_{21})$ consist of integer points with odd sum of coordinates in cones with vertices at $(\pm1, 0, 0)$ and $(0, \pm1, 0)$, one of them is shown in Figure~2. 

All roots in <<lateral>> basic sets except four <<layers>> belong to <<corner>> basic sets as well (see Figure~3). Thus, 
the set of all roots consists of four <<corner>> basic sets and four <<layers>> between them, it is shown in Figure~4. 

Note that every root belongs to at most three basic sets, this agrees with Proposition~\ref{pr_3roots}. The degrees of derivations of Type~I are the integer points in four 2-dimensional cones with vertices at $(\pm1, \pm1, 1)$, one of them is shown in Figure~5. Note that every point in this set belongs to three basic sets, but there are other points that belong to three basic sets. This agrees with Remark~\ref{rem_TI}. 
\end{example}

Let us give one more application of Theorem~\ref{theor}. It is known that for a commutative $\KK$-domain $R$ and locally nilpotent derivations $\delta_1, \delta_2$ on $R$ the condition $\Ker \delta_1 = \Ker \delta_2$ implies $h_1 \delta_1 = h_2 \delta_2$ for some $h_1, h_2 \in R$; see~\cite[Principle 12]{Fr2006}. For trinomial algebras we can prove a more precise statement. 

\begin{proposition}
Let $\R$ be a trinomial algebra and $\delta_1, \delta_2$ be finely homogeneous locally nilpotent derivations of $\R$ with $\Ker \delta_1 = \Ker \delta_2$. 
Then $\delta_i = h_i \delta$ for some locally nilpotent derivation $\delta$ of $\R$ and $h_1, h_2 \in \R$. 
\end{proposition}

\begin{proof}
By Theorem~\ref{theor}, $\delta_1$ and $\delta_2$ are elementary, that is they are of the form $h \dCb$ for some $h \in \R$. Since $h$ does not affect the kernel of $h \dCb$, it is sufficient to prove that $\dCb$ is defined by $ \Ker \dCb$ up to a constant from $\KK$. 

First note that the set of variables $T_{ij}$ with $\dCb(T_{ij}) = 0$ uniquely determines the sequence~$C$ and the type of the derivation, see Construction~\ref{dCb_constr}. 

Now let us prove that $\beta = (\beta_0, \beta_1, \beta_2)$ is defined by $\Ker \dCb$ up to a constant. 
According to Proposition~\ref{prop_ker}, we have $\beta_1 T_0^{l_0} - \beta_0 T_1^{l_1} \in \Ker \dCb$. 
Let any other $\alpha_1 T_0^{l_0} - \alpha_0 T_1^{l_1}$ belong to~$\Ker \dCb$. Suppose that the pairs $(\alpha_0, \alpha_1)$ and $(\beta_0, \beta_1)$ are non-proportional. 
Then $T_0^{l_0}$, $T_1^{l_1}$ are linear combinations of the above elements from $\Ker \dCb$ and hence belong to $\Ker \dCb$. 
Since $\Ker \dCb$ is factorially closed, this implies that $T_{ij} \in \Ker \dCb$ for any $i, j$, a contradiction. 
Thus $\alpha_1 T_0^{l_0} - \alpha_0 T_1^{l_1}$ belongs to $\Ker \dCb$ if and only if $(\alpha_0, \alpha_1)$ is proportional to $(\beta_0, \beta_1)$. Applying $\beta_2 = -\beta_0-\beta_1$ yields that $(\beta_0, \beta_1, \beta_2)$ is determined by $\Ker \dCb$ up to a constant.  
\end{proof}

\end{document}